\documentclass{amsart}
\usepackage{amsmath,amssymb}
  \usepackage{paralist}
  \usepackage{graphics} 
  \usepackage{epsfig} 
\usepackage{graphicx}  \usepackage{epstopdf}
 \usepackage[colorlinks=true]{hyperref}
\hypersetup{urlcolor=blue, citecolor=red}

\newtheorem{theorem}{Theorem}[section]
\newtheorem{corollary}{Corollary}
\newtheorem{lemma}[theorem]{Lemma}
\newtheorem{proposition}{Proposition}

\newtheorem{definition}[theorem]{Definition}

\title[Stability distributed delay]{Optimal linear stability condition for scalar differential equations with distributed delay}

\thanks{Work supported by ANR grant ProCell ANR-09-JCJC-0100-01.}

\begin{document}
\maketitle

\centerline{\scshape Samuel Bernard }
\medskip
{\footnotesize

 \centerline{Universit\'e de Lyon; CNRS UMR 5208;}
  \centerline{Universit\'e Lyon 1; Institut Camille Jordan;}
   \centerline{INRIA Team Dracula}
   \centerline{43 blvd. du 11 novembre 1918,
F-69622 Villeurbanne cedex, France.}
} 

\medskip

\centerline{\scshape Fabien Crauste}
\medskip
{\footnotesize
 \centerline{Universit\'e de Lyon; CNRS UMR 5208;}
  \centerline{Universit\'e Lyon 1; Institut Camille Jordan;}
   \centerline{INRIA Team Dracula}
   \centerline{43 blvd. du 11 novembre 1918,
F-69622 Villeurbanne cedex, France.}
}

\bigskip


\begin{abstract}
Linear scalar differential equations with distributed delays appear in the study of the local stability of nonlinear differential equations with feedback, which are common in biology and physics. Negative feedback loops tend to promote oscillations around steady states, and their stability depends on the particular shape of the delay distribution. Since in applications the mean delay is often the only reliable information available about the distribution, it is desirable to find conditions for stability that are independent from the shape of the distribution. We show here that for a given mean delay, the linear equation with distributed delay is asymptotically stable if the associated differential equation with a discrete delay is asymptotically stable. We illustrate this criterion on a compartment model of hematopoietic cell dynamics to obtain sufficient conditions for stability.
\end{abstract}

\section{Introduction}

Models of self-regulating systems often include discrete delays in the feedback loop to account for the finite time required to perform essential steps before the loop is closed. Such mathematical simplifications are especially welcome in biological applications, where knowledge about the loop steps is usually sparse. This includes maturation and growth times needed to reach reproductive age in a population \cite{hutchinson1948,m1978}, signal propagation along neuronal axons \cite{campbell2007}, and post-translational protein modifications \cite{bernard2006b, monk2003}. Introduction of a discrete delay in an ordinary differential equation can destabilize steady states and generate complex dynamics, from limit cycles to chaos \cite{kuang1993}. Although the linear stability properties of scalar equations with single discrete delays are fairly well characterized, lumping intermediate steps into a delayed term can produce broad and atypical delay distributions that deviate from discrete delays, and it is still not clear how that affects the stability of the equation \cite{campbell2009}. 

The delayed feedback differential equation of the form 
\begin{displaymath}
\dot x = F\Bigl(x, \int_0^{\infty}  x(t-\tau) d\eta(\tau) \Bigr)
\end{displaymath}
is a model paradigm in biology and physics \cite{adimy2005, atay2003, eurich2005, meyer2008, monk2003, rateitschak2007}. The first argument of $F$ is the instantaneous part of the loop and the second one, the delayed or retarded part, which closes the feedback loop. The integral is taken in the Riemann-Stieltjes sense. The function $\eta$ is a cumulative probability distribution function, it can be continuous, discrete, or a mixture of continuous and discrete elements. In most cases, the stability of the above equation is related to its linearized equation about one of its steady states $\bar x$,
\begin{align}\label{eq:x-lin-1}
\dot x & = - a x - b \int_0^{\infty}  x(t-\tau) d \eta(\tau)
\end{align}
where the constants $a$ and $b \in \mathbb{R}$ are the negatives of the derivatives of the instantaneous and the delayed parts of $F$ at $x=\bar x$,
\begin{align*}
 a & = - \frac{\partial }{\partial x} F(x,y) \Bigl|_{x=y=\bar x} \; \quad \text{and} \quad \; b  = - \frac{\partial }{\partial y}F(x,y) \Bigl|_{x=y=\bar x}.
\end{align*}

Eq. (\ref{eq:x-lin-1}) is also called a linear retarded functional differential equation. Basic theory for delay differential equations and functional differential equations can be found in \cite{bellman1963} and \cite{hale1993}. Additional applications can be found in previously mentioned references and in \cite{erneux2009, kuang1993}. 

Stability analysis of Eq. (\ref{eq:x-lin-1}), when the distribution function $\eta$ differs from the Dirac distribution, has been the subject of several works. In 1989, Boese \cite{b1989} analyzed the stability of (\ref{eq:x-lin-1}) for a Gamma distribution, and determined rather technical sufficient conditions for its asymptotic stability. Kuang \cite{k1994}, in 1994, considered a system of two differential equations with continuous distributed delay, possibly infinite. He focused on the existence of pure imaginary eigenvalues, and determined conditions for their nonexistence, obtaining sufficient conditions for the asymptotic stability of his system. In 2001, Bernard et al \cite{bernard01} considered (\ref{eq:x-lin-1}) and determined sufficient conditions for its stability, mainly in the case where the distribution is symmetric about its mean. They then conjectured that the single Dirac measure would be the most destabilizing distribution of delays for (\ref{eq:x-lin-1}). Atay \cite{atay2008} recently gave arguments in that direction. He focused on the stability of delay differential equations near a Hopf bifurcation, and for linear delay differential equations, such as (\ref{eq:x-lin-1}), he showed that if the delay has a destabilizing effect, then the discrete delay is locally the most destabilizing delay distribution.

Huang and Vandewalle \cite{hv2004} and Tang \cite{t2005} also analyzed the stability of equations similar to (\ref{eq:x-lin-1}). The first authors were interested in the numerical stability of differential equations with distributed delay, but they proposed an interesting geometrical approach to determine conditions for the stability of (\ref{eq:x-lin-1}) for a special delay distribution. Unfortunately, their method cannot be generalized to general distributions. In \cite{t2005}, Tang determined sufficient stability conditions for very general differential equations with distributed delay, but his results are very technical and not easy to handle in practice. Adimy et al \cite{adimy2005} and Crauste \cite{c2010} obtained sufficient conditions for the existence of a Hopf bifurcation when the delay density function is decreasing. In \cite{obc2008}, Ozbay et al. investigated the stability of linear systems of equations with distributed delays, and applied their results to a model of hematopoietic stem cell dynamics. Considering an exponential distribution of delays, they obtained necessary and sufficient conditions for the stability using the small gain theorem and Nyquist stability criterion. Solomon and Fridman, using Linear Matrix Inequalities, also established sufficient conditions for exponential stability of systems with infinite distributed delays \cite{sf2013}. Berezansky and Braverman recently obtained sufficient conditions for the stability of non-autonomous differential equations with distributed delay \cite{bb2011, bb2013}.  

Finally, let us mention the work of Anderson \cite{anderson1991, a1992}, who focused on the stability of some delay differential
equations, called \emph{regulator models}, which are a particular form of (\ref{eq:x-lin-1}). The theory developed by Anderson \cite{anderson1991, a1992} focuses on the properties of the probability distribution $\eta$. Although the results of Anderson are only valid for some class of probability measures, they stress the importance of the shape of the delay distribution. Moreover, Anderson mentions that ``the more concentrated the probability measure, the worse the stability property of the model'' \cite{a1992}. 

Although it has been observed that in general a greater relative variance provides a greater stability, a property linked to geometrical features of the delay distribution \cite{anderson1991}, there are counter-examples to this principle. Yet, as mentioned above, it has been conjectured that among distributions with a given mean, the discrete delay is the least stable one \cite{atay2008, bernard01}. If this were true, a theorem due to Hayes \cite{hayes1950} would provide a sufficient condition for the stability of the trivial solution of delay differential equations independently from the shape of the delay distribution. This conjecture has been proved by Krisztin using Lyapunov-Razumikhin functions when there is no instantaneous part \cite{krisztin1990}, and by different authors for distributions that are symmetric about their means \cite{atay2008,bernard01,kiss2009,miyazaki1997}. It is possible to lump the non-delayed term into the delay distribution and use the condition found in \cite{krisztin1990}, but the resulting stability condition is not optimal. Here we prove that the conjecture is true for all delay distributions with exponential tails. That is, for a given mean delay, the scalar linear differential equation with a distributed delay is asymptotically stable provided that the corresponding equation with a single discrete delay is asymptotically stable. This sufficient condition for stability is optimal in the sense that if it is not satisfied, we can find a distribution with distributed delay for which the equation is not stable. To illustrate this general result, we consider a compartment model of hemato\-poiesis that can be expressed as a scalar differential equation with an arbitrarily complex delay distribution, and we obtain a simple stability condition. 

In section \ref{s:pre}, we provide definitions and, in Section \ref{s:gsc}, we set the stage for the main stability results. In section \ref{s:d}, we show that a distribution of discrete delays is necessarily stable when the discrete distribution with a single delay equal to the mean is stable. In section \ref{s:g}, we present the generalization to any distribution, hence showing that distributions with distributed delays provide more stability than the discrete distribution with the same mean. Section \ref{s:hemato} is devoted to the presentation of a model for hematopoiesis and the illustration of the stability problem.

\section{Definitions}\label{s:pre}

We consider the linear retarded functional differential equation
\begin{align}\label{eq:x}
\dot x & = - a x - b \int_0^{\infty}  x(t-\tau) d \eta(\tau)
\end{align}
with real constants $a$ and $b$. We assume that $\eta$ is a cumulative probability distribution function: $\eta: [0,\infty) \to [0, 1]$ is monotone nondecreasing, right-continuous, $\eta(\tau)=0$ for $\tau<0$ and $\eta(+\infty)=1$. The corresponding probability density functional  $f(\tau)$ is given by the generalized derivative $d \eta(\tau) = f(\tau) d\tau$. The following definitions and Theorem \ref{thm:asympt} follow from St\'ep\'an \cite{stepan1989}. 

Let $B$ be the vector space of continuous and bounded functions on $[-\infty,0] \to \mathbb{R}$. With the norm $|| \phi || = \sup_{\theta \in [-\infty,0]} | \phi(\theta) |$, $\phi \in B$, $B$ is a Banach space.

\begin{definition} The function $x:\mathbb{R} \to \mathbb{R}$ is a solution of Eq.~(\ref{eq:x}) with the initial condition 
\begin{equation}\label{eq:ic}
x_\sigma = \phi, \; \sigma \in \mathbb{R}, \; \phi \in B, 
\end{equation} 
if there exists a scalar $\delta>0$ such that $x_t \equiv x(t+\theta) \in B$ for $\theta\in[-\infty,0]$ and $x$ satisfies Eqs. (\ref{eq:x}) and (\ref{eq:ic}) for all $t \in [\sigma, \sigma+\delta)$. 
\end{definition}

The notation $x_t(\sigma,\phi)$ is also used to refer to the solution of Eq.~(\ref{eq:x}) associated with the initial conditions $\sigma$ and $\phi$.
\begin{definition} The trivial solution $x=0$ of Eq.~(\ref{eq:x}) is stable if for every $\sigma \in \mathbb{R}$ and $\varepsilon>0$ there exists $\delta = \delta(\varepsilon)$ such that $|| x_t(\sigma,\phi) || < \epsilon$ for any $t \geq \sigma$ and for any function $\phi \in B$ satisfying $|| \phi || < \delta$. The trivial solution $x=0$ is called asymptotically stable if it is stable, and for every $\sigma \in \mathbb{R}$ there exists $\Delta = \Delta(\sigma)$ such that $\lim_{t \to \infty} || x_t(\sigma,\phi)||=0$ for any $\phi \in B$ satisfying $|| \phi || < \Delta$.
\end{definition}

\begin{definition}
The function $D:\mathbb{C} \to \mathbb{C}$ given by 
\begin{displaymath}
D(\lambda) = \lambda + a + b \int_0^\infty e^{-\lambda \tau} d\eta(\tau),
\end{displaymath}
is called the characteristic function of the linear equation (\ref{eq:x}). The equation $D(\lambda) = 0$ is called the characteristic equation of (\ref{eq:x}).
\end{definition}

The following theorem \cite{hale1974, stepan1989} gives a necessary and sufficient condition for the asymptotic stability of $x=0$.
\begin{theorem}\label{thm:asympt}
Suppose that there exists $\nu>0$ such that the following inequality is satisfied:
\begin{align}\label{eq:nu}
\int_0^{\infty} e^{\nu \tau} d\eta(\tau) < \infty.
\end{align}
The solution $x=0$ of Eq.~(\ref{eq:x}) is (exponentially) asymptotically stable if and only if all roots of the characteristic equation $D(\lambda)=0$ have $\Re(\lambda)<0$.
\end{theorem}

Theorem \ref{thm:asympt} is equivalent to the statement that solutions of Eq.~(\ref{eq:x}) of the form $x(t) = \sum_{i=1}^{\infty} \nu_i(t) e^{\lambda_i t}$ where $\lambda_i$ are the roots of the characteristic equation and $\nu_i(t)$ polynomials, are enough to determine the stability of $x=0$. Other solutions, the small solutions, decay faster than any exponential; hence the exponential stability. 

Inequality (\ref{eq:nu}) implies that the mean delay value is finite,
\begin{align*}
E & := \int_0^{\infty} \tau d\eta(\tau) < \infty.
\end{align*}
We assume in the following that inequality (\ref{eq:nu}) is always satisfied. For more details concerning retarded functional differential equations with infinite delays, see \cite{hale1974,hale1978}.

When $\eta$ represents a single discrete delay ($\eta$ a heaviside function), the asymptotic stability of the zero solution of Eq.~(\ref{eq:x}) is fully determined by the following theorem, originally due to Hayes \cite{hayes1950}.

\begin{theorem}\label{th:hayes}
Let $f(\tau) = \delta(\tau-E)$ a Dirac mass at $E$. The zero solution of Eq.~(\ref{eq:x}) is asymptotically stable if and only if $a>-b$ and $a \geq |b|$, or if $b>|a|$ and
\begin{displaymath}
E <  \frac{\arccos(-a/b)}{\sqrt{b^2-a^2}}.
\end{displaymath}
\end{theorem}
More generally, the following statements always hold for any delay distribution:
\begin{itemize}
\item[(i)] When $a \leq -b$, the characteristic equation of Eq.~(\ref{eq:x}) has a positive real root.
\item[(ii)] When $a\geq|b|$ and $a>-b$, the characteristic equation of Eq.~(\ref{eq:x}) has no root with positive real part. 
\end{itemize}
Therefore, the stability of the solution $x=0$ depends on the delay distribution only in the parameter space region $b>|a|$ and, from now on, we restrict the stability analysis to that region. 

Assuming $b>0$ and making the change of timescale $t \to bt$, we have $a \to a/b$, $b \to 1$ and $\eta(\tau) \to \eta(b\tau)$. Eq.~(\ref{eq:x}) can be rewritten as
\begin{align}\label{eq:xx}
 \dot x & = - a x - \int_0^{\infty} x(t-\tau) d\eta(\tau).
\end{align}
The delay distributions affect the stability of Eq.~(\ref{eq:xx}) when $a\in]0,1[$. 

The characteristic equation is called stable if all roots have $\Re(\lambda)<0$ \cite{stepan1989}. To emphasize the relation between the stability and the delay distribution, we give a similar definition for the delay distribution. 
\begin{definition} \label{def:stable}
The delay distribution $\eta$ (or the density $f$) is called stable if all roots of the characteristic equation of Eq.~(\ref{eq:x}), or Eq.~(\ref{eq:xx}), have $\Re(\lambda)<0$. The delay distribution $\eta$ (or the density $f$) is called unstable if there exists a characteristic root with $\Re(\lambda)>0$.     
\end{definition}

According to Theorem \ref{thm:asympt} and using Definition \ref{def:stable}, the solution $x=0$ of Eq.~(\ref{eq:xx}) is asymptotically stable if and only if the delay distribution is stable. The characteristic equation of Eq.~(\ref{eq:xx}) is
\begin{align}\label{eq:ce}
D(\lambda) = \lambda + a + \int_0^\infty{e^{-\lambda \tau} d\eta(\tau)} = 0.
\end{align} 
The integral term in Eq.~(\ref{eq:ce}) is the Laplace transform $\mathcal{L}$ of the distribution $\eta$. Along the imaginary axis $\lambda = i\omega$, the Laplace transform can be expressed as $(\mathcal{L}\eta)(i \omega) = C(\omega)-i S(\omega)$, where  
\begin{displaymath}
C(\omega) = \int_{0}^{\infty} \cos(\omega \tau) d\eta(\tau), \qquad S(\omega) = \int_{0}^{\infty} \sin(\omega \tau) d\eta(\tau).
\end{displaymath}

The strategy for determining the stability of distributed delays is the following. We use a geometric argument to bound the roots of characteric equation (\ref{eq:ce}) by the roots of the characteristic equation for a single discrete delay. More precisely, we will show that if the leading roots associated to the discrete delay are a pair of imaginary roots, then all the roots associated to the distribution of delays have negative real parts. We first state, in Section \ref{s:gsc}, a criterion for stability: if $S(\omega)<\omega$ whenever $C(\omega)=-a$, then the distribution is stable. 
We then show in Theorem \ref{th:n} that a distribution of $n$ discrete delays is more stable than a certain distribution with two delays (in the sense that $S(\omega)\leq S^*(\omega)$, where the distribution with $n$ delays is denoted by $\eta$ and the ``special'' distribution with two delays by $\eta^*$). We construct this most ``unstable'' distribution and determine that only one of the delays is positive, so that its stability can be determined using Theorem \ref{th:hayes}. We then generalize for any distribution of delays in Section \ref{s:g}.

\section{General Stability Criteria}\label{s:gsc}

Assume $a\in]-1,1[$, and let $\eta$ be a distribution with mean $E$. We consider the family of distributions, scaled with the parameter $\rho \geq 0$,
\begin{align}\label{eq:scale}
\eta_\rho(\tau) = \begin{cases} 
	\eta(\tau/\rho), & \rho>0, \\
	H(\tau), & \rho=0,
\end{cases}
\end{align}
where $H(\tau)$ is the step or heaviside function at 0, corresponding to a single discrete delay vanishing at $\tau=0$. The distribution $\eta_\rho$ has a mean $\rho E \geq 0$.  The notation $D_\rho$ is used to refer to the characteristic equation associated with the scaled distribution $\eta_\rho$. The characteristic equation for the distribution $\eta_0$ is $D_0(\lambda) := \lambda + a + 1=0$. 

The next proposition provides a necessary condition for instability. It is a direct consequence of Theorem 2.19 in \cite{stepan1989}. 

\begin{proposition}\label{pr:os} If the distribution $\eta$ is unstable, then there exists $\omega_s \in (0,\omega_c]$, $\omega_c=\sqrt{1-a^2}$, such that $C(\omega_s) = -a$ and $S(\omega_s) \geq \omega_s$.
\end{proposition}

\begin{proof} Suppose that the distribution $\eta$ is unstable, i.e.~that the characteristic equation has roots $\lambda$ with $\Re(\lambda) \geq 0$. Consider the family of scaled distributions $\eta_\rho$. The roots of the characteristic equation $D_\rho=0$ depend continuously on the parameter $\rho$ and roots with positive real parts can only appear by crossing the imaginary axis. The scaled distribution $\eta_{\rho}$ is stable for $\rho=0$ (the only root is $\lambda=-(a+1)<0$) and unstable for $\rho=1$. Hence there exists a critical value $0<\rho \leq 1$ at which $\eta_{\rho}$ loses its stability, and this happens when the characteristic equation $D_\rho(\lambda)=0$ has a pair of imaginary roots $\lambda = \pm i\omega$, with $\omega \geq 0$. Splitting the characteristic equation in real and imaginary parts, we have
\begin{equation}\label{partDrho}
\left\{\begin{array}{rcl}
\Re(D_{\rho}(i \omega)) & =  &\displaystyle\int_{0}^{\infty} \cos(\omega \tau) d\eta_{\rho}(\tau) + a = 0, \\ 
\Im(D_{\rho}(i \omega)) & = &\omega - \displaystyle\int_{0}^{\infty} \sin(\omega \tau) d\eta_{\rho}(\tau) = 0.
\end{array}\right.
\end{equation}
Since $-\omega$ satisfies the above system, we only look from now on and
throughout this manuscript to positive values of $\omega$. The upper bound on $\omega$, $\omega_c=\sqrt{1-a^2}$, is obtained by applying Cauchy-Schwartz inequality,
\begin{equation*}
a^2+\omega^2 = \Bigl( \int_{0}^{\infty} \cos(\omega \tau) d\eta_{\rho}(\tau) \Bigr)^2 + \Bigl( \int_{0}^{\infty} \sin(\omega \tau) d\eta_{\rho}(\tau) \Bigr)^2 \leq 1.
\end{equation*}
Rewriting (\ref{partDrho}) in term of $\eta$, we have
\begin{displaymath}
\int_{0}^{\infty} \cos(\omega \rho \tau) d\eta(\tau) = -a, \qquad \int_{0}^{\infty} \sin(\omega \rho \tau) d\eta(\tau) = \omega.
\end{displaymath}
Finally, setting $\omega_s:=\rho \omega$, we obtain $0 < \omega_s \leq \omega \leq \omega_c$ and
\begin{displaymath}
C(\omega_s) = \int_{0}^{\infty} \cos(\omega_s \tau) d\eta(\tau) = -a, \qquad
S(\omega_s) = \int_{0}^{\infty} \sin(\omega_s \tau) d\eta(\tau) = \omega \geq \omega_s.
\end{displaymath}
This completes the proof.
\end{proof}
Proposition \ref{pr:os} provides a sufficient condition for asymptotic stability, stated in the following corollary.
\begin{corollary}\label{th:cor}
The distribution $\eta$ is stable if one of the two following conditions is satisfied: 
\begin{itemize}
\item[(i)] $C(\omega)>-a$ for all $\omega \in [0,\omega_c]$, 
\item[(ii)] $C(\omega)=-a$, for $\omega \in \left]0,\omega_c \right]$, implies that $S(\omega)<\omega$.
\end{itemize}
\end{corollary}

The condition $S(\omega)<\omega$ is not necessary for stability, as there are cases where $S(\omega)\geq\omega$ even though the distribution is stable. This happens when an unstable distribution switches back to stability as $E$ is further increased (see \cite{beretta2002} or \cite{b1989}). 

\section{Stability of a distribution of discrete delays}\label{s:d}

In this section, we show that a distribution with $n$ discrete delays and mean $E$ is more stable than the distribution with a single discrete delay $E$. It is convenient to represent distributions of discrete delays by their densities. We denote a density of $n$ discrete delays $\tau_i \geq 0$, and weights $p_i > 0$, $i=1,...,n$, $n \geq 1$, as 
\begin{displaymath}
 f_n(\tau) = \sum_{i=1}^{n} p_i \delta(\tau-\tau_i)
 \end{displaymath}
where $\delta(\tau-\tau_i)$ is a Dirac mass at $\tau_i$, and
\begin{equation}
\sum_{i=1}^{n} p_i \tau_i = E, \quad \text{ and } \quad \sum_{i=1}^{n} p_i = 1. \label{eq:pi}
\end{equation} 
The characteristic equation associated with the density $f_n$ is $D_n(\lambda) = \lambda + a + \sum_{i=1}^n p_i e^{-\lambda \tau_i}=0$. Likewise, we denote
\begin{displaymath}
C_n(\omega)  = \sum_{i=1}^n p_i \cos(\omega \tau_i), \qquad
S_n(\omega)  = \sum_{i=1}^n p_i \sin(\omega \tau_i).
\end{displaymath}

Following Corollary \ref{th:cor}, for $f_n$ to be stable, it is enough to show that $S_n(\omega_s)<\omega_s$ whenever $C_n(\omega_s)=-a$, $\omega_s \leq \omega_c$. We now show that among all distributions satisfying $C_n(\omega_s)=-a$ for a fixed $\omega_s$, there exists a density $f^*$ that maximizes the values of $S_n(\omega_s)$. This density $f^*$ has only one positive delay, making it easy to show that $S^*(\omega_s)<\omega_s$. This would imply that all discrete delay distributions are stable. 

\begin{definition} \label{def:c}
We define the constants $c \approx 0.7246$ and $\theta_c \approx 2.3311$, where $c$ is the smallest positive value such that $\cos(\theta) \geq 1 - c \theta$ for all $\theta>0$, found by solving the two equations $c=\sin(\theta)$ and $1-\theta \sin(\theta) = \cos(\theta)$ for $c>0$, $\theta>0$, and $\theta_c$ is the positive value for which $\cos(\theta) = 1-c\theta$.  We define the convex function $g(x): [0,\pi] \to [-1,1]$ by
\begin{align*}
g(x) = \begin{cases}
1 - cx, & 0\leq x < \theta_c, \\
\cos(x), & \theta_c \leq x \leq \pi.
\end{cases}
\end{align*}
\end{definition}

Convexity implies $g(px_1+(1-p)x_2) \leq pg(x_1) + (1-p)g(x_2)$, for $p \in [0,1]$, and $x_1, x_2 \in [0,\pi]$. In addition, we have $g(x) \leq \cos(x)$. 

The following lemmas show how to find the distribution that maximizes $S_n(\omega_s)$ for $n=2$. 

\begin{lemma}\label{lem:1}
Assume $a \in \left]-1,1\right[$ and $E>0$ satisfies
\begin{align}\label{eq:sc}
E < \frac{\arccos{(-a)}}{\omega_c},
\end{align}
with $\omega_c = \sqrt{1-a^2}$. Suppose that there exists $\omega_s \in [0,\omega_c]$ and a density $f_2$ with mean $E$, such that 
\begin{equation}\label{eq:cw}
C_2(\omega_s) := p_1\cos(\omega_s\tau_1)+p_2\cos(\omega_s\tau_2) = -a. 
\end{equation}
Then $\omega_s E < \theta_c$ and $\cos(\omega_s E) > -a$.
\end{lemma}

\begin{proof}
From inequality (\ref{eq:sc}), one gets $\omega_c E < \arccos(-a) < \pi$, so $\cos(\omega_c E) > -a$. Moreover, since $\omega_c\leq 1$, the inequality $\omega_s \leq \omega_c$ implies $\cos(\omega_s E) \geq \cos(\omega_c E)$. Consequently, $\cos(\omega_s E) > -a$ and, using (\ref{eq:cw}), we then deduce that $\cos(\omega_s E) > C_2(\omega_s)$. 

Furthermore, we have
\begin{align*}
C_2(\omega_s) := p_1 \cos(\omega_s \tau_1) + p_2 \cos(\omega_s \tau_2)
  \geq p_1 g(\omega_s \tau_1) + p_2 g(\omega_s \tau_2)
  \geq g(\omega_s E).
\end{align*}
The first inequality comes from the definitions of $c$ and $g$ (see Definition \ref{def:c}): $\cos(x) \geq 1 - cx$ for $x\geq 0$. The second inequality is the convexity property of $g$. Thus, we deduce $\cos(\omega_s E) > g(\omega_s E)$. Since $g(x)=\cos(x)$ for $x \geq \theta_c$, this means that $\omega_s E < \theta_c$. 
\end{proof}


\begin{lemma}\label{th:fs} 
Assume $a \in \left]-1,1\right[$ and $E>0$ satisfies (\ref{eq:sc}). Suppose that there exists $\omega_s \in [0,\omega_c]$ and a density $f_2$ with mean $E$, such that equality (\ref{eq:cw}) is satisfied. Then there exists a unique density $f^*$ with two discrete delays $\tau_1^*$ and $\tau_2^*$, mean $E$, such that $\tau_1^*=0$ and $0 < \omega_s \tau_2^* \leq \theta_c \leq \pi$, and satisfying
\begin{align}
C_2^*(\omega_s) & = -a. \label{eq:cw*}
\end{align}
\end{lemma}

\begin{proof}
Suppose there exists a density $f_2'$ with two discrete delays $\tau_1'$ and $\tau_2'$, weights $p_1'$ and $p_2'$, mean $E$, satisfying $\tau_1' = 0$ and $\tau_2' > 0$. Necessarily, $p_2' \tau_2' = E$ (so $f_2'$ has mean $E$). We are going to show that 
\begin{equation}\label{eq:C'}
C_2'(\omega_s) = -a.
\end{equation}

By using $p_1' = 1 - p_2'$ and $p_2' = E/\tau_2'$, Eq.~(\ref{eq:C'}) is equivalent to
\begin{equation}\label{eq:cos}
\cos(\omega_s \tau_2') = 1 - \frac{1+a}{\omega_s E}  \omega_s \tau_2'.
\end{equation}
From the definitions of the constant $c$ and the function $g$ (Definition \ref{def:c}), the equation $\cos(x)=1-(1+a)x/\omega_s E$ has positive solutions in $[0,\pi]$  if and only if $\cos(x)\geq 1-cx$, that is
\begin{equation}\label{eq:cs}
c \geq \frac{1+a}{ \omega_s E}.
\end{equation}
To see that inequality (\ref{eq:cs}) is indeed satisfied, one can note that, using (\ref{eq:cw}),
\begin{align*}
-a=C_2(\omega_s) = \sum_{i=1}^2 p_i \cos(\omega_s \tau_i)  \geq \sum_{i=1}^2 p_i (1-c \omega_s \tau_i)  = 1 - c \omega_s E,
\end{align*}
so  $-a \geq 1-c \omega_s E$. Thus (\ref{eq:cs}) holds true. Consequently Eq.~(\ref{eq:cos}) has at least one solution satisfying $0\leq \omega_s \tau_2'\leq \pi$. 

Moreover, since $\theta_c$ is a tangent point (see Definition \ref{def:c}), there is exactly one solution satisfying 
$$
0\leq \omega_s \tau_2' < \theta_c < \pi.
$$
Denote by $\tau_2^*$ the smallest value of $\tau_2'$ that solves Eq.~(\ref{eq:cos}), and define $f^* = \sum_{i=1}^2 p_i^* \delta(\tau-\tau_i^*)$, with $p_2^* = E/\tau_2^*$, $p_1^* = 1 - p_2^*$, and $\tau_1^* = 0$. From the definition of $\tau_2^*$, $f^*$ exists and is unique. It remains to show that $f^*$ is a well-defined density, that is $p_2^*\in[0,1]$. Since $\tau_2^*$ is the smallest and unique positive solution in the interval $[0, \theta_c]$ of (\ref{eq:cos}), the sign of $\cos(x) - (1 - (1+a)x/(\omega_s E))$ determines whether $x$ is smaller or larger than $\tau_2^*$ in the interval $[0, \theta_c]$. From Lemma \ref{lem:1}, $\omega_s E < \theta_c$ and $\cos(\omega_s E) > -a$, or formulated equivalently, $\cos(\omega_s E) >   1 - (1+a)/(\omega_s E) \omega_s E$. Thus, $\cos(\omega_s E) -  (1 - (1+a)\omega_s E/(\omega_s E) ) > 0$, which implies that $\omega_s E < \omega_s \tau_2^*$. Since $E = p_2^* \tau_2^*$, we obtain the result $0 < p_2^* < 1$, which shows that $f^*$ is a well-defined density.
\end{proof}

\begin{lemma}\label{th:S} 
Assume $a \in \left]-1,1\right[$ and $E>0$ satisfies (\ref{eq:sc}). Suppose that there exists $\omega_s \in [0,\omega_c]$ and a density $f_2$ with mean $E$, such that equality (\ref{eq:cw}) is satisfied. Then for any density $f_2$ with mean $E$ and satisfying Eq.~(\ref{eq:cw}), we have 
\begin{align*}
S_2(\omega_s) \leq S^*(\omega_s),
\end{align*}
where the density $f^*$ is defined in Lemma \ref{th:fs}.
\end{lemma}

\begin{proof}
We recast the problem in a slightly different way. Consider a density with two discrete delays $\tau_1$ and $\tau_2$ and mean $E$, such that $C_2(\omega_s)=-a$. Writing $u=\omega_s \tau_1$, $v=\omega_s \tau_2$ and $T = \omega_s E$, we can express the weights $p_i$ in terms of $u$ and $v$:
\begin{equation*}
p_1 = \frac{v-T}{v-u} \quad\text{and}\quad p_2 = \frac{T-u}{v-u}.
\end{equation*}
By convention, $0 \leq u < T < v$. We consider $C_2(\omega_s)$ and $S_2(\omega_s)$ as functions of $u$ and $v$; $C, S: [0,T)\times (T,\infty) \to [-1,1]$ with
\begin{align}
C(u,v) & = \frac{v-T}{v-u} \cos(u) + \frac{T-u}{v-u} \cos(v), \label{eq:cuv} \\
S(u,v) & = \frac{v-T}{v-u} \sin(u)  + \frac{T-u}{v-u} \sin(v). \label{eq:suv}
\end{align}
The subscripts $2$ have been dropped to ease the reading. Equation (\ref{eq:suv}) is to be maximized with the restriction $C(u,v)=-a$. The equation $C(u,v)=-a$ defines a one-parameter family of solutions. For $u$ fixed, $C(u,v)=-a$ has at most a finite number of solutions $v$, with $v>T$. As shown in Lemma \ref{th:fs}, there is always a solution when $u=0$ (equivalent to $\tau_1=0$). Fixing $u$ and labeling the solutions in increasing order $v_i$, $i=1,...,r$, $r \geq 1$, we have $S(u,v_1) \geq S(u,v_i)$ for all $i$. Indeed, the equality $C(u,v)=-a$ can be rewritten as
\begin{equation}\label{eq:cosu}
-\frac{a+\cos(u)}{T-u} v + \frac{au + T \cos(u)}{T-u} = \cos(v).
\end{equation}
Since $0\leq u<T<\pi$, $\cos(u)$ is decreasing, $\cos(u)>\cos(T)$, and $\cos(T)>-a$ (From Lemma \ref{lem:1}, we know that $\cos(\omega_s E)>-a$), so $\cos(u)+a>0$. Eq. (\ref{eq:cosu}) writes
\begin{equation}\label{eq:cosv}
\cos(v)=\alpha(u) - \beta(u) v,
\end{equation}
where $\beta(u)>0$ for all $u\in \left[0,T \right)$, $\alpha(0)=1$ (in case $u=0$, Eq. (\ref{eq:cosv}) reduces to Eq. (\ref{eq:cos})) and $\alpha(u)$ is increasing for $u\in \left[0,T \right)$. The slope of the right hand side of (\ref{eq:cosv}) is negative, $\cos(v_i)$ is decreasing with solutions $v_i$ of (\ref{eq:cosv}) (Figure \ref{f:eqcosvi-chords}A). One may note that the points $\bigl(C(u,v_i),S(u,v_i)\bigr)$ are at the intersection of the chord $i$ between the unit circle points $\bigl(\cos(u),\sin(u)\bigr)$ and $\bigl(\cos(v_i),\sin(v_i)\bigr)$ and the vertical secant at $-a$. From (\ref{eq:cuv}) with $C(u,v_i)=-a$, it is easy to see that $\cos(v_i)<-a$ since $\cos(u)>-a$. By displaying the above mentioned chords and the vertical secant on a unit circle (Figure \ref{f:eqcosvi-chords}B), it follows that all the chords $i$, $i>1$, lie below chord 1, and thus $S(u,v_1)\geq S(u,v_i)$, $i\geq1$. 

\begin{figure}
\begin{center}A\includegraphics[trim= 1cm 0.5cm 1cm 1cm]{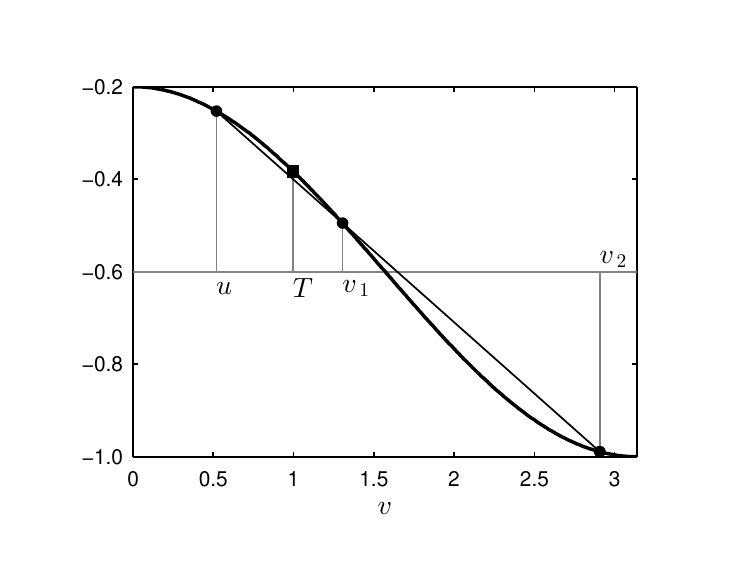}
B\includegraphics[trim= 1cm 0.5cm 1cm 1cm]{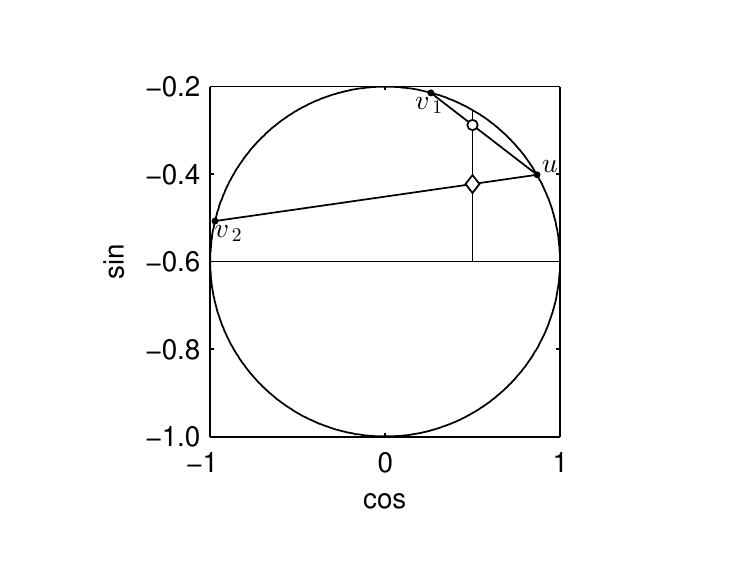}\end{center}
\caption{Two solutions in $v$ of Eq. ~(\ref{eq:cosv}) for fixed $u$, with density $f_2(\tau) = 0.8 \delta(\tau-0.625) + 0.2 \delta(\tau-3.5)$ and $a=-0.5$ (parameter chosen to satisfy the conditions of Lemma \ref{lem:1}). Then $\omega_c = \sqrt{1-a^2}=0.8660$, $E=1.2 < \arccos(-a)/\omega_c=1.2092$. Eq. ~(\ref{eq:cw}) was solved for $\omega_s=0.8308 < \omega_c$ to yield $T=\omega_s E = 0.9969$ and $u = \omega_s \tau_1 = 0.5192$). The solution $v_2 = \omega_s \tau_2 = 2.9078$ corresponds to the density $f_2$ and $v_1 = 1.3056$, to the density $f^* = 0.3925 \delta(\tau-0.625) + 0.6075 \delta (\tau - 1.5715)$.  (A) Solutions along the cosine. (B) Solutions parametrized on the circle, illustrating that at the intersection of the secant at $-a$, the value of $S(u,v_1)$ ($\circ$) is strictly larger than $S(u,v_2)$ ($\diamond$).
}\label{f:eqcosvi-chords}
\end{figure}

It is therefore enough to look, for each $u$, at the smallest solution $v_1$ of the equation $C(u,v)=-a$. The solution, which exists for $u \in \left[0, T \right)$, can be parameterized by $u$, with $v_1 = v_1(u) = \min\{v|C(u,v)=-a\}$. At $u=0$, the solution $v_1(0) = \omega_s \tau_2^*$. Therefore, we need to show that $S(0,v_1(0))$ maximizes $S(u,v_1(u))$. The total derivative of $S$ with respect to $u$ is
\begin{displaymath}
\frac{d}{d u} S(u,v_1(u)) = \frac{\partial S}{\partial u} + \frac{\partial S}{\partial v}\frac{d v_1}{d u}.
\end{displaymath}  
If $\partial S/\partial v<0$, the total derivative is strictly negative if and only if
\begin{equation}\label{eq:dvdu}
\frac{d v_1}{d u} > - \frac{\partial S}{\partial u} / \frac{\partial S}{\partial v}.
\end{equation} 
The partial derivative with respect to $v$ is
\begin{displaymath}
\frac{\partial S}{\partial v} = \frac{T-u}{v-u} \Biggl[ \frac{\sin(u) - \sin(v)}{v-u} + \cos(v) \Biggr].
\end{displaymath}
One can see that $v_1$ always satisfies $v_1(u) \leq \pi$. Indeed, if one assumes by contradiction $v_1(u)>\pi$, then first Eq. (\ref{eq:cosv}) has no root on the interval $[T, \pi]$, and second, since $\cos(T)>-a=\alpha(u)-\beta(u)T$, one gets $\cos(v)>\alpha(u)-\beta(u)v$ for $v\in[T,\pi]$. It follows that for all $v>\pi$,
$$
\alpha(u)-\beta(u)v<\alpha(u)-\beta(u)\pi<\cos(\pi)=-1\leq \cos(v),
$$
and Eq. (\ref{eq:cosv}) has no root, yielding a contradiction.

The sine function is strictly concave on the interval $[0,\pi]$ and this implies that
\begin{equation}\label{eq:conc}
\sin(u) < \sin(v) + (u-v) \frac{d}{d v}\sin(v), 
\end{equation}
or equivalently that $(\sin(u) - \sin(v))/(v-u) + \cos(v)<0$, for all $0 \leq u < v \leq \pi$. This shows that $\partial S/\partial v<0$. Now,
\begin{align*}
\frac{d v_1}{du} & = \frac{v-T}{T-u} \frac{\cos(v)-\cos(u)+(v-u)\sin(u)}{\cos(u)-\cos(v)-(v-u)\sin(v)}, \\
- \frac{\partial S}{\partial u} / \frac{\partial S}{\partial v} & = \frac{v-T}{T-u} \frac{\sin(v)-\sin(u)-(v-u)\cos(u)}{\sin(u)-\sin(v)+(v-u)\cos(v)}.
\end{align*}
Inequality (\ref{eq:dvdu}) can be re-expressed as
\begin{displaymath}
(v-u) \bigl[ 2 - 2\cos(v-u) - (v-u) \sin(v-u) \bigr] > 0.
\end{displaymath}
It can be verified that this inequality is satisfied for $v-u=z \in \left(0,\pi \right]$. The left-hand side vanishes when $z \to 0$, and the derivative is strictly positive for $0 < z \leq \pi$:
\begin{align*}
\frac{d}{dz} \bigl[ 2 - 2 \cos(z) - z \sin(z)\bigl]  = \sin(z) - z \cos(z) > 0.
\end{align*}
The last inequality is obtained with inequality (\ref{eq:conc}). Therefore, $dS/du<0$ and $S$ is maximized for $u=\omega_s \tau_1^* = 0$ and $v_1(0)=\omega_s \tau_2^*<\pi$.
\end{proof}

Now that we established the existence of a density $f^*$ with two delays, one equal to zero the other one positive, and mean $E$ which maximizes the quantity $S_2(\omega_s)$, we prove in the next theorem the stability of all densities with $n$ discrete delays and mean $E$ satisfying (\ref{eq:sc}).

\begin{theorem}\label{th:n} Assume $a \in \left]-1,1 \right[$ and $E>0$ satisfies inequality (\ref{eq:sc}). Let $f_n$ be a discrete density with $n \geq 1$ delays and mean $E$, then the density $f_n$ is stable.
\end{theorem}

\begin{proof} 
Case $n=1$. Single delay distributions ($n=1$) are stable by Theorem \ref{th:hayes}. 

Case $n=2$. Consider a density $f_2$ with two delays $\tau_1 < \tau_{2}$. If $C_2(\omega_s)>-a$ for every $\omega_s \in [0, \omega_c]$, Corollary \ref{th:cor} states that $f_2$ is stable. Suppose $C_2(\omega_s)=-a$ for a value $\omega_s \in [0,\omega_c]$. From Lemmas \ref{th:fs} and \ref{th:S}, there exists a density $f^*$ with $\tau_1^*=0$ and $0<\tau_2^*\leq \pi/\omega_s$ such that $C^*(\omega_s)=C_2(\omega_s)$ and $S^*(\omega_s) \geq S_2(\omega_s)$. 

Since $S^*$ maximizes the value of $S_2$, if we are able to show that any distribution $f^*$ with a zero and a positive delay, and $C^*(\omega_s)=-a$, satisfies $S^*(\omega_s)<\omega_s$, then from Corollary \ref{th:cor} all distributions with two delays will be stable. 

Let the density $f^*(\tau) = (1-p) \delta(\tau) + p \delta(\tau-\tau^*)$ with $p \in \left(0,1 \right]$ and $\tau^* \in [E,\pi/\omega_s]$. We have $C^*(\omega_s) = 1-p + p \cos(\omega_s \tau^*) = -a$. We must show that $S^*(\omega_s) = p \sin(\omega_s) < \omega_s$. Summing up the squares of the cosine and the sine, we then obtain
$
p^2 = (-a+p-1)^2 + S^{*2}(\omega_s),
$ 
so
$
S^*(\omega_s) = \sqrt{p^2 - (-a+p-1)^2}.
$
Since $E$ satisfies inequality (\ref{eq:sc}), then
$
\tau^* < \arccos(-a)/p\sqrt{1-a^2}.
$
From $C^*(\omega_s)=-a$ we get $\omega_s=\arccos( - (a+1-p) p^{-1})/\tau^*$. Thus,
\[
p \sqrt{1-a^2}\frac{\arccos \bigl( - (a+1-p) p^{-1} \bigr)}{\arccos(-a)} < \frac{\arccos \bigl( - (a+1-p) p^{-1} \bigr)}{\tau^*} = \omega_s.
\] 
Since $(a+1-p)p^{-1} \geq a$ for $p \in \left(0,1 \right]$ and $a \in \left]-1,1 \right[$, we have the following inequality
\[
\frac{\arccos(-a)}{\sqrt{1-a^2}} \leq \frac{\arccos \bigl( - (a+1-p) p^{-1} \bigr)}{\sqrt{1-\bigl( (a+1-p) p^{-1} \bigr)^2}},
\]
which implies
\[
p\sqrt{1-\bigl( (a+1-p) p^{-1} \bigr)^2}\leq p \sqrt{1-a^2}\frac{\arccos \bigl( - (a+1-p) p^{-1} \bigr)}{\arccos(-a)}.
\]
Thus, 
\[
S^*(\omega_s) =  \sqrt{p^2 - (-a+p-1)^2} \leq p \sqrt{1-a^2} \frac{\arccos \bigl( - (a+1-p) p^{-1} \bigr)}{\arccos(-a)} < \omega_s.
\]
This completes the proof for the case $n=2$. 

Case $n>2$. For densities $f$ with $n>2$ delays, the strategy is also to find an upper bound for the value of $S(\omega_s)$ via a new distribution $f^*$ that keeps $C(\omega_s)=-a$ constant. If, for the new distribution, $S(\omega_s) \leq S^*(\omega_s)< \omega_s$ holds true, then Corollary \ref{th:cor} can be applied. The construction of $f^*$ requires two or three steps. In the first step, all delays $\tau_i > \pi/\omega_s$ are replaced by smaller delays $\tau_i' < \pi/\omega_s$, in order to use the concavity of the sine function on the interval $[0,\pi]$ as done in the proof of Lemma \ref{th:S}, in the following way:
\begin{displaymath}
\tau_i' =\begin{cases}
\tau_i - 2 k_i \pi / \omega_s & \text{if $\sin(\omega_s \tau_i)\geq 0$},\\
2(k_i+2)\pi/\omega_s - \tau_i & \text{if $\sin(\omega_s \tau_i) < 0$},
\end{cases}
\end{displaymath}
where $k_i = \max\{j|2j\pi/\omega_s \leq \tau_i \}$. This transformation preserves $C(\omega_s)$: $\cos(\omega_s \tau_i') = \cos(\omega_s \tau_i)$, and ensures that $S(\omega_s)$ increases: $\sin(\omega_s \tau_i') = |\sin(\omega_s \tau_i)|$. That way, we obtain an associated delay density $f'$ with $C'(\omega_s)=-a$, $S'(\omega_s) \geq S(\omega_s)$, $E' \leq E$ and $\tau_i' \leq \pi/\omega_s$. 

In the second step, we reduce the number of strictly positive delays. All pairs of delay $\tau_i' < \tau_j'$ for which the inequality 
\begin{align}\label{eq:belowcos}
\frac{p_i \cos(\omega_s \tau_i')+p_j \cos(\omega_s \tau_j')}{p_i+p_j} \leq \cos\biggl(\omega_s  \frac{p_i \tau_i^{\prime}+p_j \tau_j^{\prime}}{p_i+p_j}\biggr)
\end{align}
holds are iteratively replaced by one positive and one vanishing delay, as done in Lemma \ref{th:fs}. We note that inequality (\ref{eq:belowcos}) reduces to 
\begin{displaymath}
C(\omega_s) = p_1\cos(\omega_s\tau_1) + p_2\cos(\omega_s\tau_2) \leq \cos(\omega_s E)
\end{displaymath}
for a two discrete delay distribution, with delays $\tau_1$ and $\tau_2$ satisfying (\ref{eq:pi}). This transformation preserves the values of mean $E'$ and $C'(\omega_s)$, and increases the value of $S'(\omega_s)$. This step is repeated until one of the two situations occurs: (i) There remains one density $f^*$ with exactly one delay $\tau_1^* = 0$ and one delay $\tau_2^* > 0$. Then the inequality $S^*(\omega_s)<\omega_s$ follows from the first part of the proof. Therefore, $S(\omega_s) \leq S'(\omega_s) \leq S^*(\omega_s) <\omega_s$, and, by Corollary \ref{th:cor} implies that $f$ is stable. (ii) There remains a density $\bar f$ with one delay $\bar \tau_1 = 0$ and two or more delays $\bar \tau_k > 0$, $k=2,\dots,m$, $m \geq 3$, such that
\[
\frac{\bar p_i \cos(\omega_s \bar \tau_i)+\bar p_j \cos(\omega_s \bar \tau_j)}{\bar p_i+\bar p_j} > \cos\left(\omega_s \frac{\bar p_i \bar \tau_i+\bar p_j \bar \tau_j}{\bar p_i+\bar p_j}\right),
\]  
for each pair $i \neq j \in 2,\dots,m$. Since $\sum_{k=1}^m \bar p_k =1$, the strictly positive delays now satisfy
\begin{align}
\sum_{k=2}^m \frac{\bar p_k \cos(\omega_s \bar \tau_k)}{1-\bar p_1} > \cos\left( \omega_s \sum_{k=2}^m \frac{\bar p_k \bar \tau_k}{1-\bar p_1} \right), \label{eq:sineq}
\end{align}
while $\bar C(\omega_s) := \sum_{k=1}^m \bar p_k \cos(\omega_s \bar \tau_k) = -a \leq  \cos(\omega_s \bar E)$. 

The third step is to replace all positive delays $\bar \tau_k$, $k=2,\dots,m$, with the single mean delay
\[
\tau_2'' = \sum_{k=2}^m \frac{\bar p_k \bar \tau_k}{1-\bar p_1}. 
\]
Because the sine function is concave on the interval $[0,\pi]$, any averaging of delays can only increase the value of $S$. We now have a density $f''$ with $\tau_1''=0$ and $\tau_2''>0$, $p_1''=\bar p_1$ and $p_2''=1-\bar p_1$, $C''(\omega_s) < \bar C(\omega_s)$ (from inequality (\ref{eq:sineq})), $E'' = \bar E \leq E$, and $S''(\omega_s) \geq \bar S(\omega_s)$. We now replace $\tau_2''$ with a delay $\tau_2^{*} < \tau_2''$, so as to obtain a density $f^*$ with $C^{*}(\omega_s)=\bar C(\omega_s)=-a$, and $E^*= E''$. 

Indeed, this consists in finding $(p_2^*,\tau_2^*)$ such that $p_2^*\tau_2^*=E''=p_2''\tau_2''$, $\tau_2^*<\tau_2''$, and $1-p_2^*+p_2^*\cos(\omega_s\tau_2^*)=-a$. Hence, this is equivalent to finding $\tau_2^*\in \left] E'',\tau_2'' \right[$ such that 
$$
\chi(\tau_2^*):= 1-\frac{p_2''\tau_2''}{\tau_2^*}+\frac{p_2''\tau_2''}{\tau_2^*}\cos(\omega_s\tau_2^*)=-a.
$$
Since $\chi$ is continuous, with $\chi(E'')=\cos(\omega_s E'')=\cos(\omega_s \bar{E})\geq -a$, and $\chi(\tau_2'')=C_2''(\omega_s)<-a$, there is at least one $\tau_2^*\in \left] E'',\tau_2'' \right[$ satisfying the above conditions, with $p_2^* := p_2''\tau_2''/\tau_2^*$. Moreover, since $\tau_2^*<\tau_2''$ and the function $\sin(x)/x$ is decreasing on $(0,\pi)$, one obtains, using $p_2^*\tau_2^*=E''=p_2''\tau_2''$, that $p_2^*\sin(\omega_s\tau_2^*)\geq p_2''\sin(\omega_s\tau_2'')$, or equivalently, $S^{*}(\omega_s) \geq S''(\omega_s)$.

Consequently, this last change of delay has the effect of increasing the value $S^{*}(\omega_s) \geq S''(\omega_s)$, while maintaining the condition $C^{*}(\omega_s)=-a$. Since the mean $E^*$ of density $f^{*}$ satisfies inequality (\ref{eq:sc}), we have $S^{*}(\omega_s) < \omega_s$ as shown for the case $n=2$. Therefore $S(\omega_s) \leq S'(\omega_s) \leq \bar S(\omega_s) \leq S''(\omega_s) \leq S^{*}(\omega_s) < \omega_s$. Corollary \ref{th:cor} implies that $f$ is stable. 
\end{proof}

\section{Stability of a general distribution of delays}\label{s:g}
We now show that the stability of discrete delays implies the stability of general distributions. First we need to bound the roots of the characteristic equation for general distributed delays.

\begin{lemma}\label{th:mu}
Assume $a \in \left]-1,1 \right[$ and $E>0$ satisfies inequality (\ref{eq:sc}). Let $\eta$ be a delay distribution with mean $E$ and characteristic equation $D(\lambda)=0$. There exists a sequence of distributions $\{\eta_{n}\}_{n \geq 1}$ with mean $E$, such that $\eta_{n}$ converges weakly to $\eta$ as $n\to \infty$, and $\lambda$ is a root of the characteristic equation if and only if there exists a sequence of characteristic roots $\lambda_n$ for $\eta_{n}$ such that $\lim_{n \to \infty} \lambda_n = \lambda$. If $\{\mu_n\}_{n \geq 1}$ is a sequence of real parts of characteristic roots $\lambda_n$ for $\eta_{n}$, $D_{n}(\lambda_n)=0$, then $\limsup_{n \to \infty} \mu_n < 0$.
\end{lemma}

\begin{proof}
Existence of a sequence $\{\eta_{n}\}_{n \geq 1}$ of distributions with $n$ delays and mean $E$, such that $\eta_{n}$ converges weakly to $\eta$ as $n\to \infty$ is rather straightforward, this sequence can be built explicitly. We do not detail this part here. 

Consider $\lambda_n = \mu_n + i \omega_n$ a root of the characterisitic equation for $\eta_{n}$. The mean $E$ satisfies inequality (\ref{eq:sc}), so $\mu_n < 0$. Then,
\begin{align*}
\Bigl\lvert D(\lambda_n) \Bigl\lvert & = \Bigl\lvert \lambda_n + a + \int_0^{\infty} e^{-\lambda_n \tau} d \eta(\tau) \Bigl\lvert \\
&  = \Bigl\lvert \lambda_n + a + \int_0^{\infty} e^{-\lambda_n \tau} d [\eta(\tau)-\eta_{n}(\tau)] + \int_0^{\infty} e^{-\lambda_n \tau} d \eta_{n}(\tau) \Bigr\rvert \\
& = \Bigl\lvert \int_0^{\infty} e^{-\lambda_n \tau} d [\eta(\tau)-\eta_{n}(\tau)] \Bigl\lvert \to 0,
\end{align*}
as $n \to \infty$ by weak convergence. Thus any converging sub-sequence of roots converges to a root for $\eta$. The same way, if $\lambda$ is a root for $\eta$,  
\begin{align*}
\Bigl\lvert D_{n}(\lambda) \Bigl\lvert & = \Bigl\lvert \lambda + a + \int_0^{\infty} e^{-\lambda \tau} d \eta_{n}(\tau) \Bigl\lvert \\
&  = \Bigl\lvert \lambda + a + \int_0^{\infty} e^{-\lambda \tau} d [\eta_{n}(\tau)-\eta(\tau)] + \int_0^{\infty} e^{-\lambda \tau} d \eta(\tau) \Bigr\rvert \\
& = \Bigl\lvert \int_0^{\infty} e^{-\lambda \tau} d [\eta_{n}(\tau)-\eta(\tau)] \Bigl\lvert \to 0,
\end{align*}
as $n \to \infty$. Convergence is guaranteed by inequality (\ref{eq:nu}). Thus each root $\lambda_n$ lies close to a corresponding root $\lambda$, and $\mu = \limsup_{n \to \infty} \mu_n$, with $\mu_n$ real part of a characteristic root $\lambda_n$, is the real part of a characteristic root for $\eta$. Since $\mu_n<0$, we have that $\mu$ is non-positive. Suppose $\mu=0$ and consider the scaled distribution $\eta_{a,\rho}(\tau)$ defined by (\ref{eq:scale}), and the associated real parts $\mu_{a, \rho}$, where the subscript $a$ is there to emphasize the dependence of the stability on the parameter $a$ in the characteristic equation. Then, by continuity, there exists $(\bar a, \rho)$ in an $\varepsilon$-neighborhood of the point $(a,1)$ for which $\mu_{\bar a, \rho}>0$. For sufficiently small $\varepsilon>0$, inequality (\ref{eq:sc}) is still satisfied:
\begin{equation*}
\rho E < \frac{\arccos(- \bar a)}{\sqrt{1-\bar a^2}}.
\end{equation*}
Additionally, the scaled discrete distributions $\eta_{n,\bar a, \rho}$ converge weakly to $\eta_{\bar a, \rho}$, so that the real parts $\mu_{n, \bar a, \rho}$ of the roots converging to $\mu_{\bar a, \rho}$ become eventually positive. That is, there is $N>1$ such that $\eta_{n,\bar a, \rho}$ is unstable for all $n>N$, a contradiction to Theorem \ref{th:n}, since inequality (\ref{eq:sc}) still holds. Therefore $\mu<0$. 
\end{proof}

\begin{theorem}\label{th:main} Assume $a \in \left]-1,1 \right[$ and $E>0$ satisfies inequality (\ref{eq:sc}). Let $\eta$ be a delay distribution with mean $E$, then the distribution $\eta$ is stable.
\end{theorem}

\begin{proof}
Consider a sequence of distributions with $n$ delays $\{\eta_{n}\}_{n \geq 1}$ where $\eta_{n}$ converges weakly to $\eta$. By Lemma \ref{th:mu}, the roots of the characteristic equation of $\eta$ have strictly negative real parts. Therefore $\eta$ is stable.
\end{proof}

The results obtained above provide the most complete picture of the stability of Eq.~(\ref{eq:x}) when the only information about the distribution of delays is the mean. These results are summarized in the following theorem and illustrated in Fig. \ref{f:chart}.

\begin{figure}
\begin{center}\includegraphics[trim=0.5cm 0.5cm 0.5cm 0.5cm]{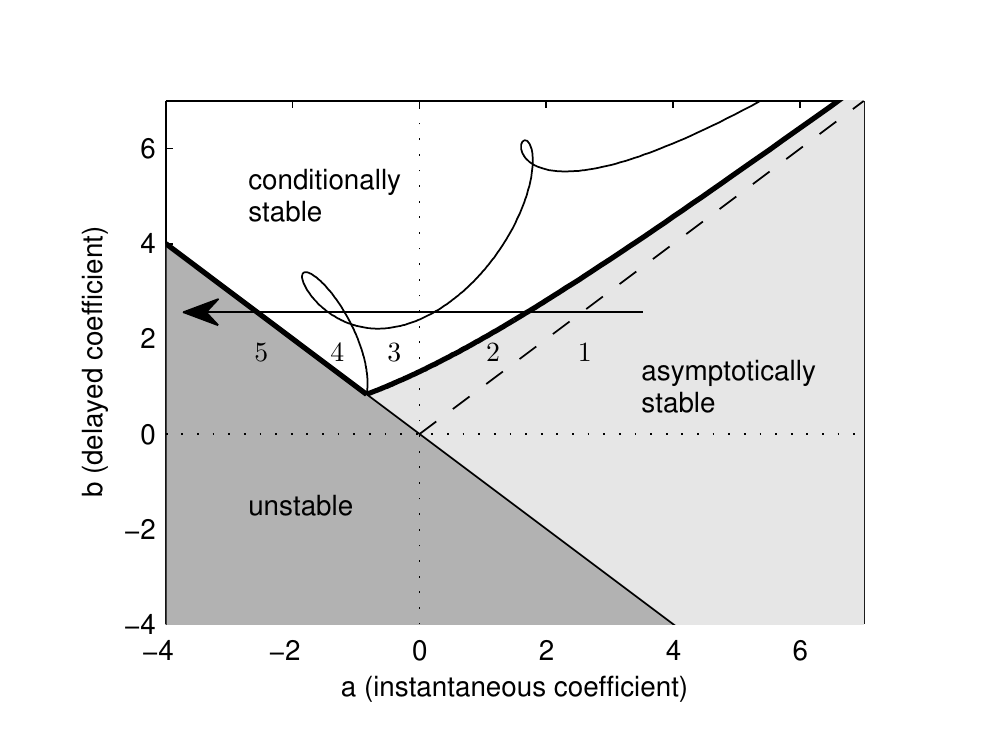}\end{center}
\caption{Stability chart of distributions of delay in the $(a,b)$ plane, obtained from Theorem \ref{th:stab}. The asymptotic stability region is composed of regions (1) to (3): a delay-independent stability region (\emph{light grey}, (1)), delimited by the condition $a \geq |b|$; a discrete-delay stability region (\emph{conditionally stable, light-grey}, (2)), delimited by condition \ref{eq:E-cor}; and a distributed-delay-dependent stability region (\emph{white}, (3)). The instability region is composed of a distributed-delay-dependent instability region (\emph{conditionally stable, white}, (4)) and a delay-independent instability region (\emph{unstable, dark grey}, (5)), delimited by the curve $b=-a$.  The discrete and distributed delay stability boundaries intersect at point $(a=-1/E,b=1/E)$.  The arrow pointing leftward shows that there exists a region, for $b>1/E$, where a stable steady state can become unstable through a decrease of the value of $a$, independently of the shape of the delay distribution. The distributed delay is $f(\tau) = 0.8 \delta(\tau-0.625) + 0.2 \delta(\tau-3.5)$, with mean delay $E=1.2$ (parameters as in Figure \ref{f:eqcosvi-chords}).}\label{f:chart}
\end{figure}

\begin{theorem}\label{th:stab}
The zero solution of Eq.~(\ref{eq:x}) is asymptotically stable if $a>-b$ and $a \geq |b|$, or if $b>|a|$ and the mean E of $\eta$ satisfies 
\begin{equation}\label{eq:E-cor}
 E < \frac{\arccos(-a/b)}{\sqrt{b^2-a^2}}.
\end{equation}
The zero solution of Eq.~(\ref{eq:x}) may not be asymptotically stable (depending on the particular distribution) if $b>|a|$ and
\begin{equation*}
 E \geq \frac{\arccos(-a/b)}{\sqrt{b^2-a^2}}.
\end{equation*}
The zero solution of Eq.~(\ref{eq:x}) is unstable if $a \leq -b$. 
\end{theorem}

\section{Compartment Model of Hematopoiesis}\label{s:hemato}

Circulating blood cells are continuously renewed by a hierarchical structure of cells maintained by hematopoietic stem cells (HSCs). Hematopoiesis consists in a complex set of feedback loops that control blood cell production. HSCs can either self-renew or differentiate to one of the three main blood cell lineages: white blood cells, platelets and red blood cells. Through successive division and differentiation stages, HSCs become progenitors (immature cells), precursors (differentiated cells), and then fully mature cells. At every stage of this hierarchy, feedback loops regulate cell differentiation, proliferation, and death. The process of red blood cell production is tightly controlled by erythropoietin, a growth factor released by the kidneys when blood oxygen is low, and whose action inhibits cell death \cite{kb1990}. Platelet production and white blood cell production processes are also controlled by growth factors (thrombopoietin \cite{k2005} and G-CSF \cite{gcsf}, respectively). It is usually thought that mature blood cells act negatively, through growth factor release, on precursors, progenitors and HSCs dynamics \cite{cm2005a,cm2005b}. 

From a modeling viewpoint, the hierarchical structure of hematopoiesis can be described by a finite system of differential equations, each equation describing the dynamics of one cell generation \cite{bmm1995, bbm2003, cm2005a, cm2005b, m1978, smc2011}. Such a view is largely accepted, both by modelers and biologists, even though mechanisms involved in cell differentiation processes are complex and there is no reason to believe that cells always go through a forward differentiation process. 

In 2005, Colijn and Mackey \cite{cm2005a, cm2005b} proposed a compartment model of hemato\-poiesis, based on previous models of hematopoietic stem cell dynamics \cite{m1978}, white blood cell dynamics \cite{bbm2003}, platelet dynamics \cite{am2007} and red blood cell dynamics \cite{bmm1995}. This model consists in a system of 4 differential equations with discrete delays. Each equation describes the number of either HSCs, red blood cells, white cells or platelets. Cells spend a finite amount of time in each of these compartments during which they mature and divide. Delays account for cell stage durations. Colijn and Mackey's model \cite{cm2005a, cm2005b} has been further justified and numerically analyzed by Colijn and Mackey \cite{cm2007} and Lei and Mackey \cite{lm2011}, who showed that it exhibits multiple steady states. Stability analysis of this model is made difficult by the presence of several discrete delays. A simpler model, based on ordinary differential equations, can then be considered, similar to the one by Stiehl and Marciniak-Czochra \cite{smc2011}. However, even in this case, the structure of the system with several compartments induces a natural delay, and the stability analysis is not straightforward.

We consider a compartment model of hematopoiesis that encompasses the main dynamical properties existing hematopoiesis models, and focus on stability conditions for this system. The compartment model can be expressed as a single equation with a general distributed delay. We showed that among all delay distributions with a given fixed mean, the distribution with a single discrete delay (that is, the delay equals the mean) is the most unstable one. Consequently we can provide a condition for the stability of the hematopoiesis model by determining when the equation with a single delay is stable.

Let denote by $x(t)$ the number of HSCs at time $t$, and by $z_i(t)$, $i=1,2,3$, the densities of circulating platelets, white cells, and red blood cells, respectively. We assume that $x$ produces the quantities $z_i$ through a linear chain process, describing the compartmental structure of each hematopoietic lineage. The number of mature cells  $z_i$ act on a negative feedback loop that represses the production of $x$. The disappearance rate of HSCs, $\alpha$, is assumed constant. The HSC production rate $P$ is a function that  depends on $x$ and a weighted average $z$ of the repressors $z_i$. Namely $z = \sum_{i=1}^3 p_iz_i$, where $p_i \geq 0$ and $\sum_{i=1}^3 p_i =1$. The HSC number $x$ is governed by the equation
\begin{equation}\label{eq:x-hem}
\dot x = P(x,z) - \alpha x.
\end{equation}
Each mature cell number $z_i(t)$, $i=1,2,3$, is assumed to be the product of a linear chain of differential equations of the type
\begin{equation}\label{eq:linchain}
\left\{\begin{array}{lcl}
\dot y_i^{(1)} & =& \beta_i\bigl(x - y_i^{(1)}\bigr), \\
\dot y_i^{(j)} & =& \beta_i\bigl(y_i^{(j-1)} - y_{i}^{(j)}\bigr), \quad j=1,\dots,q_i-1, \\
\dot z_i & = & \beta_i\bigl(y_{i}^{(q_i-1)} - z_i\bigr).
\end{array}\right.
\end{equation}
In the $i$-th hematopoietic lineage, the cell number in generation $j$-th is denoted by $y_i^{(j)}$, $j=1,\dots,q_i-1$. Mature cells $z_i$ form compartment $q_i$, and immature cells $x$ compartment $0$. System (\ref{eq:x-hem})--(\ref{eq:linchain}) describes a hierarchical structure with parallel negative feedback loops of length $q_i$, with kinetic parameters $\beta_i$, $i=1,2,3$. This situation hypothesizes that each compartment in each hematopoietic lineage depends only on the previous compartment and, except for the source term $\beta_i x$, lineages are independent from each other. 

This system is an instance of a nonlinear system with a linear subsystem \cite{cooke1982,macdonald2008}. For each lineage $i$, thanks to the usual chain trick in System (\ref{eq:linchain}), the repressors $z_i$ can be expressed in terms of the history of $x$ convoluted by a Gamma distribution,
\begin{displaymath}
z_i(t) = \int_0^\infty x(t-\tau) g(\tau,q_i,\beta_i) d \tau, \qquad \textrm{with} \qquad
g(\tau,q_i,\beta_i) = \frac{\beta_i^{q_i}}{\Gamma(q_i)} \tau^{q_i-1} e^{-\beta_i \tau}.
\end{displaymath}
When one focuses only on one hematopoietic lineage, and $z=z_i$ ($p_j=0$ for $j\neq i$), Eq. (\ref{eq:x-hem}) can be expressed as a distributed delay equation with a Gamma distribution with mean $E_i = q_i/\beta_i$ and variance $V_i = q_i/\beta_i^2$. Two limiting cases are useful to consider. When $q_i=1$, mature cells are produced directly from HSCs, and the Gamma distribution becomes an exponential distribution with parameter $\beta_i$. When $E_i = q_i/\beta_i$ is made constant and $q_i \to \infty$, the Gamma distribution converges to a Dirac mass at $E_i$. 

In addition to these three standard delay distributions, more general delay distributions are obtained by considering the above-mentioned linear parallel negative feedback loops. From System (\ref{eq:linchain}), the weighted repressor $z(t)$ remains a delayed version of $x(t)$,
\begin{equation}\label{eq:z-p}
z(t) = \int_0^\infty x(t-\tau) f_p(\tau) d \tau,
\end{equation}
where the density of the distributed delay is a weighted average of Gamma densities,
\begin{displaymath}
f_{p}(\tau) = \sum_{i=1}^3 p_i g(\tau,q_i,\beta_i).
\end{displaymath}
The delay has a mean $E_p = \sum_{i=1}^3 p_i q_i/\beta_i$. In the limiting case where the length $q_i$ of each loop becomes infinite while keeping the ratio $q_i/\beta_i$ constant, the distribution becomes a combination of discrete delays. Therefore, by a suitable choice and number of parallel negative feedback loops, one can obtain an arbitrary complex distribution of delays. 

After expressing the repressor $z$ as a function of the history of $x$ in (\ref{eq:z-p}), one can then write the following equation for $x$, from (\ref{eq:x-hem}) and (\ref{eq:z-p}),
\begin{equation}\label{eq:x-hem-2}
\dot x = P\left(x, \int_0^\infty x(t-\tau) f_p(\tau) d \tau\right) - \alpha x.
\end{equation}
The dynamics of System (\ref{eq:x-hem})--(\ref{eq:linchain}) is entirely contained in (\ref{eq:x-hem-2}). Although the production term depends continuously on the history of $x$, the initial conditions need only to be known at a finite number of locations. Analyzing the stability of Eq. (\ref{eq:x-hem-2}) is however as difficult as the stability of the System (\ref{eq:x-hem})-(\ref{eq:linchain}). 

As a nonlinear production term $P$, we consider the case of a mixed feedback loop, observed when a repressor (mature cells) and an activator (immature cells) are competing. The nonlinear term in equation (\ref{eq:x-hem-2}) is  then
\begin{equation}
P(x,z) =  \frac{k_0 x^r }{1+z^h}.  \label{eq:fmix}
\end{equation}
The parameter $r$ is related to the degree of cooperativity of the positive loop. For $r>1$ the positive loop is positively cooperative and multiple stable steady states are possible. When $r=1$ the positive loop is neutrally cooperative and at most one positive steady state exists.  For $0 \leq r <1$, the positive loop is negatively cooperative and there is a single positive steady state. When $r=0$, the dependence on $x$ of the production rate $P$ is lost. To ensure solutions are bounded, we set $r \leq h$. The parameter $h$ is the Hill coefficient describing the degree of cooperativity of the negative loop. The higher the value of $h$, the steeper the negative control. We assume $h>1$. With these conditions, there is always at least one steady state $\bar x \geq 0$.

Eq. (\ref{eq:x-hem-2}) linearized around a positive steady state $\bar x > 0$ is
\begin{align}\label{eq:linmix}
\dot x = -\alpha(1-r)x - \frac{\alpha^2h}{k_0}\bar x^{h-r+1}\int_0^\infty x(t-\tau) f_p(\tau) d \tau.
\end{align}
For positive cooperativity ($1<r\leq h$), there is a stable steady state $\bar x_0=0$. In addition, there are either zero, one or two positive steady states given by the roots of the equation $\alpha \bar x^h-k_0\bar x^{r-1}+\alpha=0$. In terms of Eq. (\ref{eq:x-lin-1}), $a = \alpha(1-r) < 0$ and $b(\bar x) = \alpha^2h\bar x^{h-r+1}/k_0>0$.
The smaller positive steady state $\bar x_1$ satisfies $a\leq -b(\bar x_1)$ and, by Theorem \ref{th:stab}, is always unstable. The larger steady state $\bar x_2$ satisfies $a> -b(\bar x_2)$ and the sufficient condition on stability of Theorem \ref{th:stab} can be applied in the following proposition.

\begin{proposition}[positive cooperativity]\label{prop:2}
Assume $P$ is given by (\ref{eq:fmix}) and $1<r\leq h$ (mixed feedback loop with positive cooperativity). 
When they exist and are distinct, the smaller positive steady state $\bar x_1$ of (\ref{eq:x-hem-2}) is unstable, and the larger positive steady state $\bar x_2$ is linearly asymptotically stable if
\begin{equation}\label{eq:stabE}
E _p:= \sum_{i=1}^3 p_i \frac{q_i}{\beta_i} < \frac{\arccos \left(\frac{(r-1) k_0}{\alpha h (\bar x_2)^{h-r+1}} \right)}{\alpha \sqrt{(\alpha h(\bar x_2)^{h-r+1}/k_0)^2-(r-1)^2}}.
\end{equation}
When $\bar x_1 = \bar x_2$, the positive steady state is unstable. The zero steady state $\bar x_0 = 0$ is always linearly stable. 
\end{proposition}

For negative cooperativity ($0 \leq r < 1$), there exists a steady state $\bar x_0=0$ only if $r>0$, in which case it is unstable. In addition, there is a unique positive steady state given by the root of the equation $\alpha(1+\bar x^h)\bar x^{1-r} = k_0$. The linear equation is given by equation (\ref{eq:linmix}), and the instantaneous coefficient is $a = \alpha(1-r)> 0$, the delayed coefficient is $b = \alpha^2h\bar x^{h-r+1}/k_0 > 0$. 

For neutral cooperativity ($r=1$), there is a steady state $\bar x_0=0$, whose stability depends on the existence of a positive steady state. There exists a positive steady state $\bar x=((k_0-\alpha)/\alpha)^{1/h}$ only if $k_0>\alpha$, and in this case $a = 0$ and $b = \alpha h (k_0-\alpha)/k_0 > 0$. Theorem \ref{th:stab} can be applied in the following proposition to determine stability .

\begin{proposition}[neutral and negative cooperativity]\label{prop:3}
Assume $P$ is given by (\ref{eq:fmix}). When $r=1$ (mixed feedback loop with neutral cooperativity), a unique positive steady state, $\bar x=((k_0-\alpha)/\alpha)^{1/h}$, of (\ref{eq:x-hem-2}) exists if $k_0>\alpha$. If it exists, it is linearly asymptotically stable if 
\begin{displaymath}
E _p:= \sum_{i=1}^3 p_i \frac{q_i}{\beta_i} < \frac{k_0\pi}{2\alpha h (k_0-\alpha)}.
\end{displaymath}
The zero steady state $\bar x_0 = 0$ is stable if $k_0<\alpha$ and unstable otherwise.

When $0\leq r< 1$ (negative cooperativity), a unique positive steady state $\bar x$ of (\ref{eq:x-hem-2}) exists. It is linearly asymptotically stable if $\bar x^{h+1-r}\leq k_0(1-r)/(\alpha h)$, or if $\bar x^{h+1-r}> k_0(1-r)/(\alpha h)$ and (\ref{eq:stabE}) holds true. The zero steady state $\bar x_0 = 0$ is always unstable.
\end{proposition}

A model with neutral cooperativity has been considered before by Mackey and Glass \cite{mackey1977} in the context of blood cell production. Neutral cooperativity arises when HSCs proliferate at a rate proportional to their number. In this situation, the steady state can be solved explicitly and the stability condition is relatively simple to state. The existence condition defines whether stem cells reproduce quickly enough to maintain their population ($k_0>\alpha$) or not. The original Mackey-Glass equation contained a single discrete delay at $E_p$. Replacing the discrete delay by a general delay distribution cannot make the positive steady state unstable, as illustrated in Fig.~\ref{f:mg}.

\begin{figure}
\begin{center}
A\includegraphics[trim= 1cm 0.5cm 0.6cm 1cm]{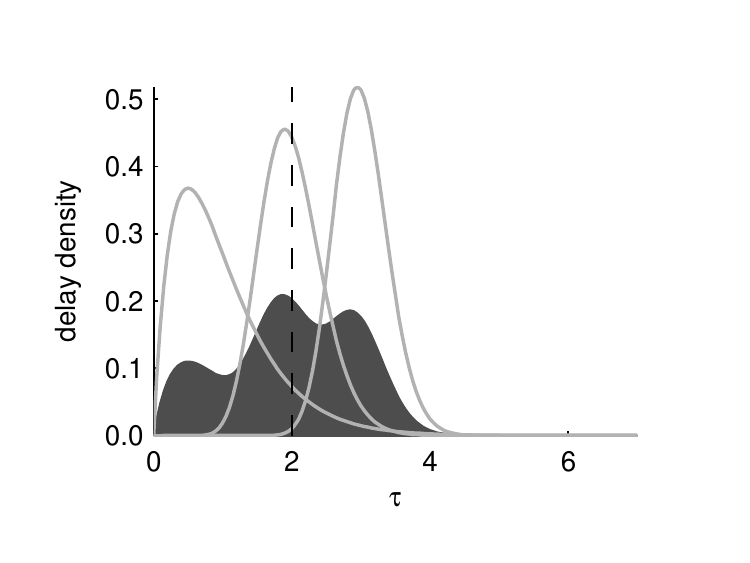}
B\includegraphics[trim= 1cm 0.5cm 0.6cm 1cm]{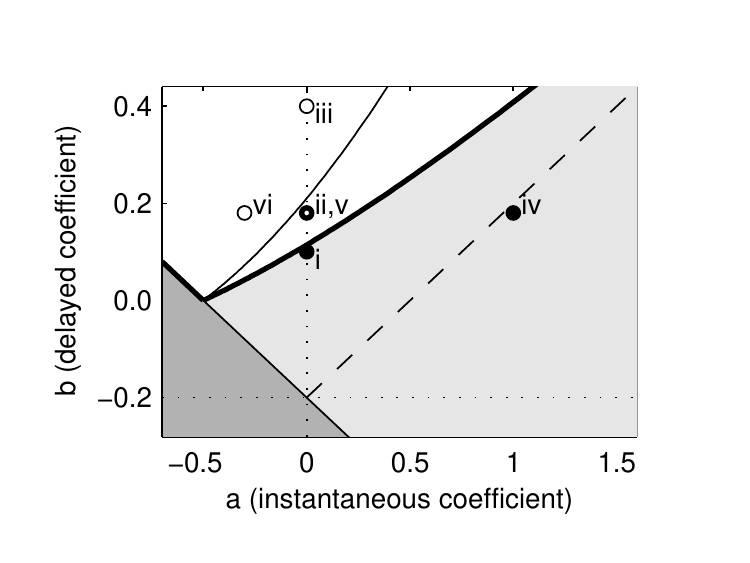}
C\includegraphics[trim= 1cm 0.5cm 0.6cm 0.5cm]{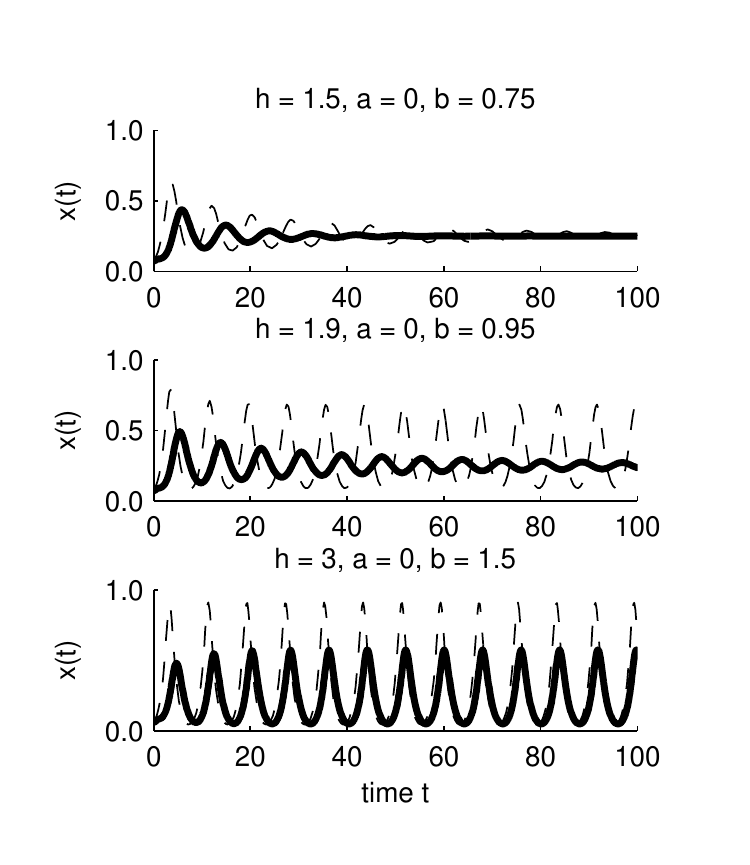}
D\includegraphics[trim= 1cm 0.5cm 0.6cm 0.5cm]{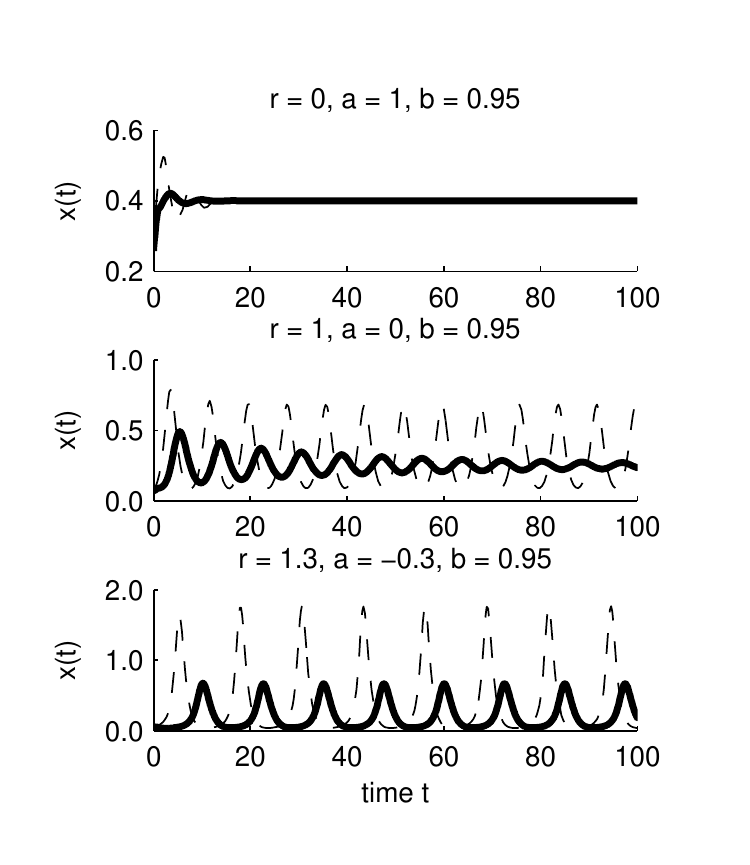}
\end{center}
\caption{Comparison of solutions of Eq.~(\ref{eq:x-hem-2}) with a distributed delay and with a discrete delay for varying values of $r$ and $h$. Fixed parameters values are $\alpha = 1$ and $k_0 = 2$ (so $\overline{x}=1$ is a steady state of (\ref{eq:x-hem-2}), whatever the values of $r$ and $h$), and $q=\{2, 20, 60\}$, $\beta = \{2, 10, 20\}$, $p = \{0.3, 0.4, 0.3\}$. (A) The distributed delay (\emph{shaded area}) is an average of three Gamma densities  (\emph{grey lines}) with mean delay $E = q/\beta = \{1, 2, 3\}$. The discrete delay is the mean delay $E_p = \sum_{i=1}^3 p_i E_i = 2$ (\emph{dashed}). (B) Stability chart of the positive steady state $\bar x = 1$, for varying $r$ and $h$, and the stability condition is given in Proposition \ref{prop:3}. The distributed delay is stable at points i, ii and iv, while the discrete delay is stable at points i and iv. Color coding is as in Fig.~\ref{f:chart}. (C, D) Time series of the system with distributed (\emph{solid}) or discrete delay (\emph{dashed}). (C) Neutral cooperativity, increasing Hill coefficient: $r=1$ and $h = 1.5$ (i), $1.9$ (ii) and $3.0$ (iii). (D) Constant Hill coefficient, increasing cooperativity: $h=1.9$ and $r = 0$ (iv), $1$ (v) and $1.3$ (vi).}\label{f:mg}
\end{figure}
 

\section{Conclusion}

We have shown that for a given mean delay, the scalar linear differential equation with a distributed delay is asymptotically stable provided that the corresponding equation with a single discrete delay is asymptotically stable. Hence, linear systems with a discrete delay are ``more'' unstable than linear systems with distributed delay. This result provides a sufficient condition for the stability of a large class of linear systems, as instanced by a model of hematopoiesis with parallel lineages. 

Quite often the aim of the modeling is not to reproduce stability but rather instability, via periodic oscillations. Pathological cases in hematopoiesis (blood diseases, leukemias) can for instance often be explained by the destabilization of the steady state which starts oscillating periodically. Our result shows that it is more difficult to reproduce periodic oscillations, observed experimentally, with a distributed delay than with a discrete delay.

\section*{Acknowledgments} 
This work has been supported by ANR grant ProCell ANR-09-JCJC-0100-01.


\bibliographystyle{siam}
\bibliography{dde}

\begin{thebibliography}{10}

\bibitem{adimy2005}
{\sc M.~Adimy, F.~Crauste, and S.~Ruan}, {\em {A mathematical study of the
  hematopoiesis process with applications to chronic myelogenous leukemia}},
  SIAM J. Appl. Math., 65 (2005), pp.~1328--1352.

\bibitem{anderson1991}
{\sc R.~Anderson}, {\em {Geometric and probabilistic stability criteria for
  delay systems}}, Math. Biosci., 105 (1991), pp.~81--96.

\bibitem{a1992}
\leavevmode\vrule height 2pt depth -1.6pt width 23pt, {\em Intrinsic parameters
  and stability of differential-delay equations.}, J. Math. Anal. Appl., 163
  (1992), pp.~184--199.

\bibitem{am2007}
{\sc R.~Apostu and M.~Mackey}, {\em {Understanding Cyclical Thrombocytopenia: A
  mathematical modeling approach}}, J. Theor. Biol., 251 (2008), pp.~297--316.

\bibitem{atay2003}
{\sc F.~Atay}, {\em {Distributed delays facilitate amplitude death of coupled
  oscillators}}, Phys. Rev. Lett., 91 (2003), p.~94101.

\bibitem{atay2008}
\leavevmode\vrule height 2pt depth -1.6pt width 23pt, {\em {Delayed feedback
  control near Hopf bifurcation}}, Discrete Contin. Dynam. Systems Ser. S, 1
  (2008), pp.~197--205.

\bibitem{gcsf}
{\sc S.~Basu, A.~Dunn, and A.~Ward}, {\em {G-CSF: function and modes of
  action}}, Int. J. Mol. Med., 10 (2002), pp.~3--10.

\bibitem{bmm1995}
{\sc J.~B\'elair, M.~C. Mackey, and J.~M. Mahaffy}, {\em Age-structured and
  two-delay models for erythropoiesis}, Math. Biosci., 128 (1995),
  pp.~317--346.

\bibitem{bellman1963}
{\sc R.~Bellman and K.~Cooke}, {\em {Differential-difference equations}},
  Academic press, 1963.

\bibitem{beretta2002}
{\sc E.~Beretta and Y.~Kuang}, {\em {Geometric stability switch criteria in
  delay differential systems with delay dependent parameters}}, SIAM J. Math.
  Anal., 33 (2002), pp.~1144--1165.

\bibitem{bb2011}
{\sc L.~Berezansky and E.~Braverman}, {\em Stability of linear differential
  equations with a distributed delay}, Comm. Pure Appl. Math., 10 (2011),
  pp.~1361--1375.

\bibitem{bb2013}
\leavevmode\vrule height 2pt depth -1.6pt width 23pt, {\em Stability of
  equations with a distributed delay, monotone production and nonlinear
  mortality}, Nonlinearity, 26 (2013), pp.~2833--2849.

\bibitem{bernard01}
{\sc S.~Bernard, J.~B\'{e}lair, and M.~C. Mackey}, {\em Sufficient conditions
  for stability of linear differential equations with distributed delay},
  Discrete Contin. Dynam. Systems Ser. B, 1 (2001), pp.~233--256.

\bibitem{bbm2003}
{\sc S.~Bernard, J.~Belair, and M.~C. Mackey}, {\em Oscillations in cyclical
  neutropenia: New evidence based on mathematical modeling}, J. Theor. Biol.,
  223 (2003), pp.~283--298.

\bibitem{bernard2006b}
{\sc S.~Bernard, B.~{\v{C}}ajavec, L.~Pujo-Menjouet, M.~Mackey, and H.~Herzel},
  {\em {Modelling Transcriptional Feedback Loops: The Role of Gro/TLE1 in Hes1
  Oscillations}}, Philos. Trans. R. Soc. London, Ser. A,  (2006),
  pp.~1155--1170.

\bibitem{b1989}
{\sc F.~Boese}, {\em The stability chart for the linearized cushing equation
  with a discrete delay and gamma-distributed delays.}, J. Math. Anal. Appl.,
  140 (1989), pp.~510--536.

\bibitem{campbell2007}
{\sc S.~Campbell}, {\em {Time delays in neural systems}}, in Handbook of Brain
  Connectivity, A.~McIntosh and V.~Jirsa, eds., Springer, 2007, pp.~65--90.

\bibitem{campbell2009}
{\sc S.~Campbell and R.~Jessop}, {\em {Approximating the Stability Region for a
  Differential Equation with a Distributed Delay}}, Math. Mod. Nat. Phenom., 4
  (2009), pp.~1--27.

\bibitem{cm2005a}
{\sc C.~Colijn and M.~Mackey}, {\em {A mathematical model of hematopoiesis --
  I. Periodic chronic myelogenous leukemia}}, J. Theor. Biol., 237 (2005),
  pp.~117--132.

\bibitem{cm2005b}
\leavevmode\vrule height 2pt depth -1.6pt width 23pt, {\em {A mathematical
  model of hematopoiesis -- II. Cyclical neutropenia}}, J. Theor. Biol., 237
  (2005), pp.~133--146.

\bibitem{cm2007}
\leavevmode\vrule height 2pt depth -1.6pt width 23pt, {\em Bifurcation and
  bistability in a model of hematopoietic regulation}, SIAM J. App. Dynam.
  Sys., 6 (2007), pp.~378--Ð394.

\bibitem{cooke1982}
{\sc K.~L. Cooke and Z.~Grossman}, {\em Discrete delay, distributed delay and
  stability switches}, J. Math. Anal. Appl., 86 (1982), pp.~592--627.

\bibitem{c2010}
{\sc F.~Crauste}, {\em Stability and hopf bifurcation for a first-order delay
  differential equation with distributed delay}, Complex Time-Delay Systems,
  (2010), pp.~263--296.

\bibitem{erneux2009}
{\sc T.~Erneux}, {\em {Applied delay differential equations}}, Springer Verlag,
  2009.

\bibitem{eurich2005}
{\sc C.~Eurich, A.~Thiel, and L.~Fahse}, {\em {Distributed delays stabilize
  ecological feedback systems}}, Phys. Rev. Lett., 94 (2005), p.~158104.

\bibitem{hale1974}
{\sc J.~Hale}, {\em {Functional differential equations with infinite delays}},
  J. Math. Anal. Appl., 48 (1974), pp.~276--283.

\bibitem{hale1978}
{\sc J.~Hale and J.~Kato}, {\em {Phase space for retarded equations with
  infinite delay}}, Funkcial. Ekvac, 21 (1978), pp.~11--41.

\bibitem{hale1993}
{\sc J.~Hale and S.~Verduyn~Lunel}, {\em {Introduction to functional
  differential equations}}, Berlin: Springer, 1993.

\bibitem{hayes1950}
{\sc N.~Hayes}, {\em {Roots of the transcendental equation associated with a
  certain difference-differential equation}}, J. Lond. Math. Soc., 25 (1950),
  pp.~226--232.

\bibitem{hv2004}
{\sc C.~Huang and S.~Vandewalle}, {\em An analysis of delay-dependent stability
  for ordinary and partial differential equations with fixed and distributed
  delays.}, SIAM J. Sci. Comput., 25 (2004), pp.~1608--1632.

\bibitem{hutchinson1948}
{\sc G.~Hutchinson}, {\em {Circular causal systems in ecology}}, Ann. N.Y.
  Acad. Sci., 50 (1948), pp.~221--246.

\bibitem{k2005}
{\sc K.~Kaushansky}, {\em {The molecular mechanisms that control
  thrombopoiesis}}, The Journal of Clinical Investigation, 115 (2005),
  pp.~3339--3347.

\bibitem{kiss2009}
{\sc G.~Kiss and B.~Krauskopf}, {\em {Stability implications of delay
  distribution for first-order and second-order systems}}, Discrete Contin.
  Dynam. Systems Ser. B, 13 (2010), pp.~327--345.

\bibitem{kb1990}
{\sc M.~Koury and M.~Bondurant}, {\em {Erythropoietin retards DNA breakdown and
  prevents programmed death in erythroid progenitor cells}}, Science, 248
  (1990), pp.~378--381.

\bibitem{krisztin1990}
{\sc T.~Krisztin}, {\em {Stability for functional differential equations and
  some variational problems}}, Tohoku Math. J, 42 (1990), pp.~407--417.

\bibitem{kuang1993}
{\sc Y.~Kuang}, {\em {Delay differential equations: With applications in
  population dynamics}}, Academic Pr, 1993.

\bibitem{k1994}
\leavevmode\vrule height 2pt depth -1.6pt width 23pt, {\em Nonoccurrence of
  stability switching in systems of differential equations with distributed
  delays.}, Quart. Appl. Math., LII (1994), pp.~569--578.

\bibitem{lm2011}
{\sc J.~Lei and M.~Mackey}, {\em Multistability in an age-structured model of
  hematopoisis: Cyclical neutropenia}, J. Theor. Biol., 270 (2011),
  pp.~143--153.

\bibitem{macdonald2008}
{\sc N.~MacDonald}, {\em Biological delay systems: linear stability theory},
  Cambridge University Press, 2008.

\bibitem{m1978}
{\sc M.~C. Mackey}, {\em Unified hypothesis of the origin of aplastic anaemia
  and periodic hematopoiesis}, Blood, 51 (1978), pp.~941--956.

\bibitem{mackey1977}
{\sc M.~C. Mackey and L.~Glass}, {\em Oscillation and chaos in physiological
  control systems}, Science, 197 (1977), pp.~287--289.

\bibitem{meyer2008}
{\sc U.~Meyer, J.~Shao, S.~Chakrabarty, S.~Brandt, H.~Luksch, and R.~Wessel},
  {\em {Distributed delays stabilize neural feedback systems}}, Biol. Cybern.,
  99 (2008), pp.~79--87.

\bibitem{miyazaki1997}
{\sc R.~Miyazaki}, {\em {Characteristic equation and asymptotic behavior of
  delay-differential equation}}, Funkcial. Ekvac., 40 (1997), pp.~481--482.

\bibitem{monk2003}
{\sc N.~Monk}, {\em {Oscillatory expression of Hes1, p53, and NF-$\kappa$B
  driven by transcriptional time delays}}, Curr. Biol., 13 (2003),
  pp.~1409--1413.

\bibitem{obc2008}
{\sc H.~Ozbay, C.~Bonnet, and J.~Clairambault}, {\em Stability analysis of
  systems with distributed delays and application to hematopoietic cell
  maturation dynamics}, in Decision and Control, 2008. CDC 2008. 47th IEEE
  Conference on, IEEE, pp.~2050--2055.

\bibitem{rateitschak2007}
{\sc K.~Rateitschak and O.~Wolkenhauer}, {\em {Intracellular delay limits
  cyclic changes in gene expression}}, Math. Biosci., 205 (2007), pp.~163--179.

\bibitem{sf2013}
{\sc O.~Solomon and E.~Fridman}, {\em New stability conditions for systems with
  distributed delays}, Automatica J. IFAC, 49 (2013), pp.~3467--3475.

\bibitem{stepan1989}
{\sc G.~St{\'e}p{\'a}n}, {\em {Retarded dynamical systems: stability and
  characteristic functions}}, Longman Scientific \& Technical New York, 1989.

\bibitem{smc2011}
{\sc T.~Stiehl and A.~Marciniak-Czochra}, {\em Characterization of stem cells
  using mathematical models of multistage cell lineages}, Math. Comp. Models.,
  53 (2011), pp.~1505--1517.

\bibitem{t2005}
{\sc X.~Tang}, {\em Asymptotic behavior of a differential equation with
  distributed delays.}, J. Math. Anal. Appl., 301 (2005), pp.~313--335.

\end{thebibliography}

\end{document}